\documentclass{article}
\usepackage{amssymb,amsmath,amsthm,verbatim}
\usepackage{ifthen}

\newtheorem{theorem}{Theorem}

\newtheorem{lemma}[theorem]{Lemma}

\newtheorem{claim}[theorem]{Claim}
\newtheorem{corollary}[theorem]{Corollary}
\newtheorem{conjecture}{Conjecture}

\theoremstyle{definition}
\newtheorem{definition}{Definition}

\theoremstyle{remark}
\newtheorem{remark}{Remark}
\newtheorem{example}{Example}

\newcommand{\mythmname}{}
\newtheoremstyle{mytheorem}
	{3pt}
	{3pt}
	{\it}
	{}
	{}
	{{\bf .}}
	{.5em}
	{\mythmname{\ifthenelse{ \equal{#3}{} }{}{\ (\thmnote{#3})}}}
\theoremstyle{mytheorem}
\newtheorem{namedtheorem}{Name}
\newcommand{\renametheorem}[1]
{
	\renewcommand{\mythmname}{{\bf #1}}
}

\newcommand{\tensor}{\otimes}

\newcommand{\GA}{{\rm GA}}

\newcommand{\EA}{{\rm EA}}
\newcommand{\TA}{{\rm TA}}

\newcommand{\IA}{{\rm IA}}
\newcommand{\GL}{{\rm GL}}

\newcommand{\D}{{\rm D}}
\renewcommand{\P}{{\rm P}}
\newcommand{\GP}{{\rm GP}}

\newcommand{\pprime}{{\prime \prime}}

\newcommand{\Venereau}{V\'en\'ereau }
\newcommand{\Vtype}{V\'en\'ereau-type }

\DeclareMathOperator{\Aut}{Aut}

\DeclareMathOperator{\Spec}{Spec}
\DeclareMathOperator{\id}{id}

\newcommand{\IC}{\mathbb{C}}

\newcommand{\IN}{\mathbb{N}}

\newcommand{\IZ}{\mathbb{Z}}

\usepackage{enumitem}

\begin{document}

\title{Strongly residual coordinates over $A[x]$}

\author{Drew Lewis \\ University of Alabama \\ amlewis@as.ua.edu}

\maketitle

\begin{abstract}
For a domain $A$ of characteristic zero, a polynomial $f \in A[x]^{[n]}$ is called a {\em strongly residual coordinate} if $f$ becomes a coordinate (over $A$) upon going modulo $x$, and $f$ becomes a coordinate (over $A[x,x^{-1}]$) upon inverting $x$.  We study the question of when a strongly residual coordinate in $A[x]^{[n]}$ is a coordinate, a question closely related to the Dolgachev-Weisfeiler conjecture.  It is known that all strongly residual coordinates are coordinates for $n=2$ .  We show that a large class of strongly residual coordinates that are generated by elementaries over $A[x,x^{-1}]$ are in fact coordinates for arbitrary $n$, with a stronger result in the $n=3$ case.  As an application, we show that all \Vtype polynomials are 1-stable coordinates.
\end{abstract}

\section{Introduction}

Let $A$ (and all other rings) be a commutative ring with one.  An {\em $A$-coordinate} (if $A$ is understood, we simply say {\em coordinate}; some authors prefer the term {\em variable}) is a polynomial $f \in A^{[n]}$ for which there exist $f_2,\ldots, f_n \in A^{[n]}$ such that $A[f,f_2,\ldots,f_n]=A^{[n]}$.  It is natural to ask when a polynomial is a coordinate; this question is extremely deep and has been studied for some time.  There are several longstanding conjectures giving a criteria for a polynomial to be a coordinate:

\begin{conjecture}[Abhyankar-Sathaye]\label{AS}
Let $A$ be a ring of characteristic zero, and let $f \in A^{[n]}$.  If $A^{[n]}/(f) \cong A^{[n-1]}$, then $f$ is an $A$-coordinate.
\end{conjecture}

\begin{conjecture}[Dolgachev-Weisfeiler] \label{DW}
Suppose $A=\IC^{[r]}$, and let $f \in A^{[n]}$.  If $A[f] \hookrightarrow A^{[n]}$ is an affine fibration, then $f$ is an $A$-coordinate.
\end{conjecture}

\begin{conjecture}Let $A$ be a ring of characteristic zero, and let $f \in A^{[n]}$.  If $f$ is a coordinate in $A^{[n+m]}$ for some $m>0$, then $f$ is a coordinate in $A^{[n]}$.
\end{conjecture}

The Abhyankar-Sathaye conjecture is known only for $A$ a field and $n=2$ (due to Abhyankar and Moh \cite{AM} and Suzuki \cite{Suzuki}, independently).  The $n=2$ case of the Dolgachev-Weisfeiler conjecture follows from results of Asanuma \cite{Asanuma} and Hamann \cite{Hamann}.  The case where both $n=3$ and $A=\IC$ follows from a theorem of Sathaye \cite{Sathaye}; see \cite{FreudDaigle} for more details on the background of the Dolgachev-Weisfeiler conjecture.

There are several examples of polynomials satisfying the hypotheses of these conjectures whose status as a coordinate remains open.  Many are constructed via a slight variation of the following classical method for constructing exotic automorphisms of $A^{[n]}$: let $x \in A$ be a nonzero divisor.  One may easily construct elementary automorphisms (those that fix $n-1$ variables) of $A_x ^{[n]}$; then, one can carefully compose these automorphisms (over $A_x$) to produce an endomorphism of $A^{[n]}$.  It is a simple application of the formal inverse function theorem to see that such maps must, in fact, be automorphisms of $A^{[n]}$.  The well known Nagata map arises in this manner:
\begin{align*}
\sigma &= (y+x(xz-y^2),z+2y(xz-y^2)+x(xz-y^2)^2) \\&= (y,z+\frac{y^2}{x}) \circ (y+x^2z,z) \circ (y,z-\frac{y^2}{x})
\end{align*}
While the Nagata map is generated over $\IC[x,x^{-1}]$ by three elementary automorphisms, Shestakov and Umirbaev \cite{SU} famously proved that it is wild (i.e. not generated by elementary and linear automorphisms) as an automorphism of $\IC[x,y,z]$ over $\IC$.  

When interested in producing exotic polynomials, we may relax the construction somewhat; let $y$ be a variable of $A^{[n]}$, and compose elementary automorphisms of $A_x^{[n]}$ until the resulting map has its $y$-component in $A^{[n]}$.  For example, the \Venereau polynomial $f=y+x(xz+y(yu+z^2))$ arises as the $y$-component of the following automorphism over $\IC[x,x^{-1}]$
\begin{align}
\phi &= (y+x^2z,z,u) \circ (y,z+\frac{y(yu+z^2)}{x},u-\frac{2z(yu+z^2)}{x}-y(yu+z^2)^2) \label{vpoly}
\end{align}

This type of construction motivates the following definition:
\begin{definition}
A polynomial $f \in A[x] ^{[n]}$ is called a {\em strongly residual coordinate} if $f$ is a coordinate over $A[x,x^{-1}]$ and if $\bar{f}$, the image modulo $x$, is a coordinate over $A$.
\end{definition}

The \Venereau polynomial is perhaps the most widely known example of a strongly residual coordinate that satisfies the hypotheses of the three conjectures (with $A=\IC[x]$), yet it is an open question whether it is a coordinate (see \cite{VThesis}, \cite{KVZ}, \cite{Freud}, and \cite{Vtype}, among others, for more on that particular question).

One may observe that the second automorphism in the above composition \eqref{vpoly} is essentially the Nagata map, and is wild over $\IC[x,x^{-1}]$.  The wildness of this map is a  crucial difficulty in resolving the status of the \Venereau polynomial.  Our present goal is to show that a large class of strongly residual coordinates generated by maps that are elementary over $\IC[x,x^{-1}]$ are coordinates.  Our methods are quite constructive and algorithmic, although the computations can become unwieldy quite quickly.  One application is to show that all \Vtype polynomials, a generalization of the \Venereau polynomial studied by the author in \cite{Vtype}, are one-stable coordinates (coming from the fact that the Nagata map is one-stably tame).  Additionally, we also very quickly recover a result of Russell (Corollary \ref{russell}) on coordinates in 3 variables over a field of characteristic zero.


\section{Preliminaries}
Throughout, we set $R=A[x]$ and $S=R_x=A[x,x^{-1}]$.  We adopt the standard notation for automorphism groups of the polynomial ring $A^{[n]}=A[z_1,\ldots,z_n]$:
\begin{enumerate}
\item $\GA_n(A)$ denotes the general automorphism group $\Aut _{\Spec A} (\Spec A^{[n]})$, which is antiisomorphic to $\Aut _A A^{[n]}$ (some authors choose to define it as the latter).
\item $\EA_n(A)$ denotes the subgroup generated by the elementary automorphisms; that is, those fixing $n-1$ variables.
\item $\TA_n(A)=\langle \EA_n(A), \GL_n(A) \rangle$ is the tame subgroup.
\item $\D_n(A) \leq \GL_n(A)$ is the subgroup of diagonal matrices.
\item $\P_n(A) \leq \GL_n(A)$ is the subgroup of permutation matrices.
\item $\GP_n(A) = \D_n(A)\P_n(A) \leq \GL_n(A)$ is the subgroup of generalized permutation matrices.
\end{enumerate}
We also make one non-standard definition when working over $R=A[x]$:
\begin{enumerate}[resume]
\item $\IA _n(R) = \{\phi \in \GA _n(R)\ |\ \phi \equiv \id \pmod{x}\}$ is the subgroup of all automorphisms that are equal to the identity modulo $x$.  It is the kernel of the natural map $\GA_n (R) \rightarrow \GA_n(A)$.
\end{enumerate}
\begin{remark}In fact, the surjection $\GA_n(R) \rightarrow \GA_n(A)$ splits (by the natural inclusion), so we have $\GA_n(R) \cong \IA_n(R) \rtimes \GA_n(A)$.
\end{remark}

\begin{definition}
Let $f_1, \ldots f_m \in R^{[n]}$.
\begin{enumerate}
\item $(f_1,\ldots,f_m)$ is called a {\em partial system of coordinates} (over $R$) if there exists $g_{m+1}, \ldots, g_n \in R^{[n]}$ such that $(f_1,\ldots,f_m,g_{m+1},\ldots,g_n) \in \GA_n(R)$.
\item $(f_1, \ldots, f_m)$ is called a {\em partial system of residual coordinates} if $R[f_1, \ldots, f_m] \hookrightarrow R^{[n]}$ is an affine fibration; that is, $R^{[n]}$ is flat over $R[f_1,\ldots,f_m]$ and for each prime ideal $\mathfrak{p} \in \Spec{R[f_1,\ldots,f_m]}$, $R^{[n]} \tensor _{R[f_1,\ldots,f_m]} \kappa(\mathfrak{p}) \cong \kappa(\mathfrak{p})^{[n-m]}$.
\item $(f_1,\ldots,f_m) $ is called a {\em partial system of strongly $x$-residual coordinates} if $(f_1,\ldots,f_m)$ is a partial system of coordinates over $S$ and $(\bar{f_1},\ldots,\bar{f_m})$, the images modulo $x$, is a partial system of coordinates over $A=\bar{R}=R/xR$.  If $x$ is understood, we may simply say {\em strongly residual coordinate}.
\end{enumerate}
A single polynomial is called a {\em coordinate} (respectively {\em residual coordinate}, {\em strongly residual coordinate}) when $m=1$ in the above definitions.
\end{definition}

\begin{remark}
If $A$ is a field, then strongly residual coordinates are residual coordinates.
\end{remark}

In light of this definition, the Dolgachev-Weisfeiler conjecture can be stated in this context as
\begin{conjecture}
Partial systems of residual coordinates are partial systems of coordinates
\end{conjecture}
Similarly, we have
\begin{conjecture}Partial systems of strongly residual coordinates are partial systems of coordinates.
\end{conjecture}

Our main focus will be on constructing and identifying strongly residual coordinates that are coordinates, although in some cases our methods will generalize slightly to partial systems of coordinates.  While we lose some generality as compared to considering residual coordinates, we are able to use some very constructive approaches.  We first give a short, direct proof of the $n=2$ case (for coordinates) that shows the flavor of our methods:
\begin{theorem}\label{n2}Let $A$ be an integral domain of characteristic zero, and $R=A[x]$.  Let $f \in R^{[2]}$ be a strongly residual coordinate.  Then $f$ is a coordinate.
\end{theorem}
\begin{proof}
Since $\bar{f}$ is a coordinate in $\bar{R}^{[2]}=\bar{R}[y,z]$, without loss of generality we may assume $f=y+xQ$ for some $Q \in R[y,z]$.  Since $f$ is an $S$-coordinate, perhaps after composing with a linear map, we obtain some $\phi =(y+xQ,z+x^{-t}P) \in \GA_2(S)$ with $J\phi=1$ and $P \in R^{[2]} \setminus xR^{[2]}$.  We inductively show that such a map $\phi$ is elementarily (over $S$) equivalent to a map with $t\leq0$, which gives an element of $\GA_2(R)$.  We compute 
$$J\phi = J(y,z)+xJ(Q,z)+x^{1-t}J(Q,P)+x^{-t}J(y,P)$$ 
Since $J\phi=1$, we have $xJ(Q,z)+x^{1-t}J(Q,P)+x^{-t}J(y,P)=0$.  Thus, comparing $x$-degrees, we must have $J(y,P) \in xR^{[2]}$.  This means $P=P_0(y)+xP_1$ for some $P_1 \in R^{[2]}$.  Then we have $(y,z-x^{-t}P_0(y)) \circ \phi = (y+xQ, z+x^{-t+1}P^\prime)$ for some $P^\prime \in R^{[2]}$ by Taylor's formula, allowing us to apply the inductive hypothesis.
\end{proof}
\begin{remark}Analogous results for residual coordinates are due to Kambayashi and Miyanishi \cite{KambayashiMiyanishi} and Kambayashi and Wright \cite{KambayashiWright}.\end{remark}

The $n=3$ case remains open, with the \Venereau polynomial providing the most widely known example of a strongly residual coordinate whose status as a coordinate has not been determined.

We next describe some notation necessary to state the most general form of our results.  

\begin{definition}
Given $\tau = (t_1,\ldots,t_n) \in \IN ^n$, define $A_\tau = R^{[m]}[x^{t_1}z_1,\ldots,x^{t_n}z_n]$.  We also set $A_\tau[\hat{z}_k]=A_\tau \cap R^{[m+n]}[\hat{z}_k]=R^{[m]}[x^{t_1}z_1,\ldots,\widehat{x^{t_k}z_k},\ldots,x^{t_n}z_n]$.
\end{definition}

Given $\tau \in \IN ^n$ and $\phi \in \GA _n (R^{[m]})$, we will consider the natural action $$\phi ^\tau := (x^{-t_1}z_1,\ldots,x^{-t_n}z_n) \circ \phi \circ (x^{t_1}z_1,\ldots,x^{t_n}z_n)$$  Note that algebraically, the image of this action this gives us the group $\Aut _{R^{[m]}} A_{\tau}$; we denote the corresponding automorphism group of $\Spec A_{\tau}$ by  $\GA _n ^{\tau} (R^{[m]}) \leq \GA _n (S^{[m]})$.  For any subgroup $H \leq \GA _n (R^{[m]})$, we analogously define $H^\tau = \{\phi ^\tau \ |\ \phi \in H\} \leq \GA _n ^{\tau}(R^{[m]})$.  We will concern ourselves mostly with $\EA _n ^\tau (R^{[m]})$, $\GL _n ^\tau(R^{[m]})$, $\GP _n ^\tau (R^{[m]})$, and $\IA _n ^{\tau}(R^{[m]})$.

We also define, choosing variables $R[y_1,\ldots,y_m]=R^{[m]}$, 
$$\IA _{m+n} ^\tau (R) := \IA _{m+n}^{ ({\mathbf 0},\tau)}(R) = \left\{ (y_1,\ldots,y_m,x^{-t_1}z_1,\ldots,x^{-t_n}z_n) \circ \phi \circ (y_1, \ldots, y_m,x^{t_1}z_1,\ldots,x^{t_n}z_n) \ |\ \phi \in \IA _{m+n}(R)\right\}$$ 
where $({\mathbf 0},\tau)=(0,\ldots,0,t_1,\ldots,t_n) \in \IN ^{m+n}$.  Note that $\IA _{m+n} ^\tau (R) \subset \IA _n ^\tau (R^{[m]})$.

Automorphisms in these subgroups can be characterized by the following lemma.

\begin{lemma} \label{characterize}
Let $\tau =(t_1,\ldots,t_n) \in \IN ^n$.
\begin{enumerate}
\item Let $\alpha \in \IA _{m+n} ^\tau(R)$.  Then there exist $F_1, \ldots, F_m, G_1, \ldots, G_n \in A_\tau$ such that
$$\alpha = (y_1+xF_1,\ldots,y_m+xF_m,z_1+x^{-t_1+1}G_1,\ldots,z_n+x^{-t_n+1}G_n)$$
\item Let $\Phi \in \EA _n ^\tau (R^{[m]})$ be elementary. Then there exists $P(\hat{z}_k) \in A_\tau[\hat{z}_k]$ such that
$$\Phi = (z_1,\ldots,z_{k-1},z_k+x^{-t_{k}}P(\hat{z}_k),z_{k+1},\ldots,z_n)$$
\item Let $\gamma \in \GL _n ^\tau (R^{[m]})$.  Then there exists $a_{ij} \in R^{[m]} \setminus xR^{[m]}$ such that
$$\gamma = (a_{11}z_1+a_{12}x^{t_2-t_1}z_2+\cdots+a_{1n}x^{t_n-t_1}z_n,\ldots, a_{1n}x^{t_1-t_n}z_1+\cdots+a_{n-1,n}x^{t_{n-1}-t_n}z_{n-1}+a_{nn}z_n)$$
\end{enumerate}
\end{lemma}

The rest of the paper is organized as follows: the most general form of our results is given in Main Theorems 1 and 2 in the next section.  Here, we state a couple of less technical versions that are easier to apply.  This section concludes with some more concrete applications of these results.  The subsequent section consists of a series of increasingly technical lemmas culminating in the two Main Theorems in section \ref{maintheorems}.

\begin{theorem}\label{AT1}
Let $\phi \in \EA_n(S^{[m]})$, and write $\phi = \Phi _0 \circ \cdots \circ \Phi _q$ as a product of elementaries. For $0 \leq i \leq q$ define $\tau _i \in \IN ^n$ to be minimal such that $(\Phi _i \circ \cdots \circ \Phi _q)(A_{\tau _i}) \subset R^{[m+n]}$.  Let $\alpha \in \IA _{n+m} ^{\tau _0}(R)$, and set $\theta = \alpha \circ \phi$.  Suppose also that either 
\begin{enumerate}
\item $A$ is an integral domain of characteristic zero and $n=2$, or 
\item $\Phi _i \in \EA _n ^{\tau _i} (R^{[m]})$ for $0 \leq i \leq q$
\end{enumerate}
Then $(\theta(y_1), \ldots, \theta(y_m))$ form a partial system of coordinates over $R$.  Moreover, if $A$ is a regular domain and $\alpha \in \TA_{m+n}(S)$, then $(\theta(y_1), \ldots, \theta(y_m))$ can be extended to a stably tame automorphism over $R$.
\end{theorem} 
\begin{proof}
If we assume hypothesis 1, the theorem follows immediately from Main Theorem 2.  If we instead assume the second hypthesis, we need only to show that $\tau _0 \geq \cdots \geq \tau _q$, as then the result follows from Main Theorem 1. Let $i <q$.  Since $\Phi _i \in \EA_n ^{\tau _i}(R^{[m]})$, we have $\Phi _i (A_{\tau _i})=A_{\tau_i}$.  Then $(\Phi _{i} \circ \cdots \circ \Phi _q)(A_{\tau _i})=(\Phi _{i+1} \circ \cdots \circ \Phi _q)(A_{\tau _i}) \subset R^{[m+n]}$.  Then the minimality assumption on $\tau _{i+1}$ immediately implies $\tau _i \geq \tau _{i+1}$ as required.
\end{proof}

It is often more practical to rephrase the general ($n > 2$) case in the following way:
\begin{theorem}\label{AT2}
Let $\phi \in \EA _n (S^{[m]})$, and write $\phi=\Phi _0 \circ \cdots \circ \Phi _q$ as a product of elementaries.  Set $\sigma _{q+1}={\bf 0} \in \IN ^n$, and for $0 \leq i \leq q$ define $\sigma _i \in \IN ^n$ to be minimal such that $\Phi _i (A _{\sigma _i}) \subset A_{\sigma _{i+1}}$.  Let $\alpha \in \IA _{n+m} ^{\sigma _0}(R)$, and set $\theta = \alpha \circ \phi$.  Then $(\theta(y_1), \ldots, \theta(y_m))$ is a partial system of coordinates over $R$.  Moreover, if $A$ is a regular domain and $\alpha \in \TA _{m+n}(S)$, then $(\theta(y_1), \ldots, \theta(y_m))$ can be extended to a stably tame automorphism over $R$.
\end{theorem}
\begin{proof}
The following two facts are immediate from the definition of $\sigma _i$:
\begin{enumerate}
\item $\Phi _i \in \EA _n ^{\sigma _i}(R^{[m]})$
\item $\sigma _0 \geq \cdots \geq \sigma _q$ 
\end{enumerate}
Once these are shown, we can apply Main Theorem 1 to achieve the result.  To see these two facts, write $\sigma _i = (s_{i,1} , \ldots, s_{i,n})$.  Without loss of generality, suppose $\Phi _i$ is elementary in $z_1$, and write
$$\Phi _i =(z_1+x^{-s}P(x^{s_{i+1,2}}z_2,\ldots,x^{s_{i+1,n}}z_n),z_2,\ldots,z_n)$$ 
for some $P (\hat{z}_1)\in A_{\sigma _{i+1}}[\hat{z}_1] \setminus xA_{\sigma _{i+1}}[\hat{z}_1]$.  Clearly, the minimality condition on $\sigma _i$ guarantees $s_{i,k}=s_{i+1,k}$ for $k=2,\ldots,n$.  Since $\Phi _i (x^{s_{i,1}}z_1)=x^{s_{i,1}}z_1+x^{s_{i,1}-s}P(\hat{z}_1) \in A_{\sigma _{i+1}} \setminus xA_{\sigma _{i+1}}$, we see $s_{i,1} \geq s_{i+1,1}$ (giving $\sigma _i \geq \sigma _{i+1}$) and $s \leq s_{i,1}$.  From the latter, one easily sees that $\Phi _i = (x^{-s_{i,1}}z_1,\ldots,x^{-s_{i,n}}z_n) \circ (z_1+x^{s_{i,1}-s}P(z_2,\ldots,z_n),z_2,\ldots,z_n) \circ (x^{s_{i,1}}z_1,\ldots,x^{s_{i,n}}z_n) \in \EA _n ^{\sigma _i} (R^{[m]})$.

\end{proof}

The remainder of this section is devoted to consequences of these three theorems in more concrete settings.
%
%
%

\begin{example}Let $m=1$ and $n=1$.  Set 
\begin{align*}
\alpha  &= (y+x^2z,z) & \Phi _0 &= (y,z-\frac{y^2}{x})
\end{align*}
Theorem \ref{AT2} implies  $(\alpha \circ \Phi _0)(y) = y+x(xz-y^2)$ is a coordinate.  The construction produces the Nagata map
$$\sigma = (y+x(xz-y^2),z+2y(xz-y^2)+x(xz-y^2)^2)$$
\end{example}

\begin{example}Let $m=1$, $n=1$, and $R=k[x,t]$.  Set
\begin{align*}
\alpha  &= (y+x^2z,z) & \Phi _0 &= (y,z+\frac{yt}{x})
\end{align*}
Theorem \ref{AT2} implies $y+x(xz+yt)$ is a coordinate.  The construction produces Anick's example
$$\beta = (y+x(xz+yt),z-t(xz+yt))$$
\end{example}

In \cite{Vtype}, a generalization of the \Venereau polynomial called \Vtype polynomials were studied by the author.  They are polynomials of the form $y+xQ(xz+y(yu+z^2),x^2u-2xz(yu+z^2)-y(yu+z^2)^2) \in \IC[x,y,z,u]$ where $Q \in \IC[x]^{[2]}$.  Many \Vtype polynomials remain as strongly residual coordinates that have not been resolved as coordinates.  However, we are able to show them all to be 1-stable coordinates, generalizing Freudenburg's result \cite{Freud} that the \Venereau polynomial is a 1-stable coordinate\footnote{Our construction provides a different coordinate system than Freudenburg's.}.
\begin{corollary}Every \Vtype polynomial is a 1-stable coordinate.
\end{corollary}
\begin{proof}
Let $Q \in \IC[x][xz,x^2u]$, and set 
\begin{align*}
\alpha &= (y+xQ,z,u,t) \\
\Phi _0 &= (y,z+yt,u,t) & \Phi _3 &= (y,z-yt,u,t)\\
\Phi _1 &= (y,z,u-2zt-yt^2,t) & \Phi _4 &= (y,z,u-2zt+yt^2)\\
\Phi _2 &= (y,z,u,t+\frac{yu+z^2}{x})
\end{align*}
A direct computation shows that $(\alpha \circ \Phi _0 \circ \cdots \circ \Phi _4)(y)=y+xQ(xz+y(yu+z^2),x^2u-2xz(yu+z^2)-y(yu+z^2)^2)$ is an arbitrary \Vtype polynomial.  
We compute the induced $\sigma$-sequence $(1,2,1) \geq (0,2,1) \geq (0,0,1) \geq (0,0,0) \geq (0,0,0)$, and note that since $Q \in A_{\sigma _0}$, then $\alpha _0 \in \IA_{4} ^{\sigma _0} (\IC[x])$.  It then follows immediately from Theorem \ref{AT2} that any \Vtype polynomial is a $\IC[x]$ coordinate in $\IC[x][y,z,u,t]$.
\end{proof}


The following result is first due to Russell \cite{Russell76}, and later appeared also in \cite{BhatRoy}.
\begin{corollary}\label{russell}
Let $k$ be a field, and let $P \in k[x,y,z]$ be of the form $P=y+xf(x,y)+\lambda x^sz$ for some $s \in \IN$, $\lambda \in k^*$ and $f \in k[x,y]$.  Then $P$ is a $k[x]$-coordinate.
\end{corollary}
\begin{proof}
Here $R=k[x]$ and $S=k[x,x^{-1}]$.  Let $\theta = (y+\lambda x^sz,z) \circ (y,z+\lambda ^{-1} x^{1-s}f(x,y)) \in \EA_3(S)$.  Then Theorem \ref{AT1} yields $\theta(y)$ is a $k[x]$-coordinate, and one easily checks that $\theta(y)=P$.
\end{proof}

\section{Main results}

Due to the tedious nature of some of these calculations, the reader is advised to first simply read the statements of the results in section \ref{calculations}, and return for the details after reading the proofs of the main theorems (section \ref{maintheorems}).

\subsection{Calculations} \label{calculations}

We proceed by detailing a series of (increasingly technical) lemmas that will aid in the proofs of the main theorems.  First, a straightforward application of Taylor's formula yields the following.
\begin{lemma}\label{taylor}
Let $\tau \in \IN ^n$, $P \in A_{\tau}$.
\begin{enumerate}
\item If $\phi \in \GA _n ^{\tau}(R^{[m]})$, then $\phi(A_{\tau})=A_{\tau}$.
\item If $\alpha \in \IA _{m+n} ^\tau (R)$, then $\alpha(P)-P \in xA_\tau$.
\end{enumerate}
\end{lemma}

Next, we note that $\GA _n ^\tau (R^{[m]})$ is contained in the normalizer of $\IA _{m+n} ^\tau (R)$ in $\GA _{m+n}(S)$.  This is slightly more general than the fact that $\IA _n ^\tau (R^{[m]}) \triangleleft \GA _n ^\tau(R^{[m]})$.
\begin{lemma}\label{alphalemma}
Let $\tau  \in \IN^n$.  Then $\IA_{m+n} ^{({\mathbf 0},\tau)}(R) \triangleleft \GA _{m+n} ^{({\mathbf 0},\tau)}(R)$.  In particular, for any $\alpha \in \IA _{m+n} ^\tau (R)$ and $\phi \in \GA _n ^\tau (R^{[m]})$, we have $\phi ^{-1}  \circ \alpha \circ \phi  \in \IA _{m+n}^\tau (R)$.
\end{lemma}
\begin{proof}
Simply note that the surjection $R=A[x] \rightarrow A$ induces a short exact sequence
\begin{equation}\label{shortexact}
0 \rightarrow \IA _{m+n} ^{({\mathbf 0},\tau)}(R) \rightarrow \GA_{m+n} ^{({\mathbf 0},\tau)}(R) \rightarrow \GA_{m+n} ^{({\mathbf 0},\tau)} (A) \rightarrow 0
\end{equation}
Here, we are viewing $\GA_{m+n}(A) \leq \GA_{m+n}(R)$ by extension of scalars, and thus obtaining $\GA _{m+n} ^{({\mathbf 0},\tau)}A \leq \GA _{m+n} ^{({\mathbf 0},\tau)}A$.
\end{proof}

\begin{corollary}\label{alphapush}
Let $\tau \in \IN^n$, $\alpha \in \IA_{m+n} ^\tau (R)$, and $\phi \in \GA _n ^\tau (R^{[m]})$.  Then there exists $\alpha ^\prime \in \IA _{m+n} ^\tau(R)$ such that $\alpha \circ \phi = \phi  \circ \alpha ^\prime$.
\end{corollary}

\begin{lemma}\label{strongIA}
Let $\tau \in \IN ^n$ and $\alpha \in \IA _{m+n} ^\tau(R)$.  Then there exists $\phi \in \EA _n ^\tau(R^{[m]}) \cap \IA _n ^\tau (R^{[m]})$ such that $$ \phi \circ \alpha \in \bigcap _{{\mathbf 0} \leq \sigma \leq \tau} \IA _{m+n} ^\sigma (R)$$
\end{lemma}
\begin{proof}
We begin by writing $\tau = (t_1,\ldots,t_n) \in \IN ^n$ and 
\begin{equation}\label{strongIAalphaeq}
\alpha = (y_1+xF_1,\ldots,y_m+xF_m,z_1+x^{-t_1+1}Q_1,\ldots,z_n+x^{-t_n+1}Q_n)
\end{equation}
for some $F_1,\ldots,F_m,Q_1, \ldots, Q_n \in A_\tau$. We prove the following by induction.
\begin{claim}
For any $\sigma^\prime = (s_1,\ldots,s_n) \in \IN ^n$, there exists $\phi \in \EA _n ^\tau(R^{[m]}) \cap \IA _n ^\tau (R^{[m]})$ such that $\phi \circ \alpha \in \IA _{m+n} ^\tau (R)$ is of the form
$$\phi \circ \alpha = (y_1+xF_1,\ldots,y_m+xF_m,z_1+xz_1G_1+x^{-t_1+s_1+1}H_1,\ldots, z_n+xz_nG_n+x^{-t_n+s_n+1}H_n) $$
for some $G_1,H_1,\ldots,G_n,H_n \in A_\tau$.
\end{claim}
Clearly the case $\sigma ^\prime = \tau$ proves the lemma, since $\sigma \leq \tau$ implies $A_{\tau} \subset A_{\sigma}$.  We induct on $\sigma ^\prime$ in the partial ordering of $\IN ^n$.  Our base case of $\sigma^\prime =(0,\ldots,0)$ is provided by $\phi = \id$ (from \eqref{strongIAalphaeq}).  So we assume $\sigma ^\prime > {\mathbf 0}$.

Suppose the claim holds for $\sigma ^\prime$.  We will show that this implies the claim for $\sigma ^\prime+e_k$, where $e_k$ is the $k$-th standard basis vector of $\IN ^n$.  Without loss of generality, we take $k=1$, so $e_1=(1,0,\ldots,0)$.  By the inductive hypothesis, we may write

$$\alpha ^\prime := \phi \circ \alpha = (y_1+xF_1,\ldots,y_m+xF_m,z_1+xz_1G_1+x^{-t_1+s_1+1}H_1,\ldots, z_n+xz_nG_n+x^{-t_n+s_n+1}H_n)$$
for some $G_i,H_i \in A_\tau$ and $\phi \in \EA _n ^\tau(R^{[m]})\cap \IA _n ^\tau (R^{[m]})$.  Write $H_1=P(\hat{z}_1)+x^{t_1}z_1Q$ for some $Q \in A_\tau$ and $P(\hat{z}_1) \in A_\tau[\hat{z}_1]$.  Then we may set $\phi ^\prime = (z_1-x^{-t_1+s_1+1}P(\hat{z}_1),z_2,\ldots,z_n) \in \EA _n ^\tau (R^{[m]}) \cap \IA _n ^\tau (R^{[m]})$ and compute
\begin{align}
(\phi ^\prime \circ \alpha ^\prime)(z_1)  &=z_1+xz_1G_1+x^{-t_1+s_1+1}(H_1-\alpha^\prime(P(\hat{z}_1))) \notag   \\
&= z_1+xz_1G_1+x^{-t_1+s_1+1}(P(\hat{z}_1)+x^{t_1}z_1Q-\alpha^\prime(P(\hat{z}_1))\notag \\
&=z_1+xz_1(G_1+x^{s_1}Q)+x^{-t_1+s_1+1}(P(\hat{z}_1)-\alpha^\prime(P(\hat{z}_1)))\label{eq1} 
\end{align}
Since $\alpha ^\prime \in \IA _{m+n} ^\tau(R)$, we can write (by Lemma \ref{taylor}) $\alpha ^\prime (P(\hat{z}_1))=P(\hat{z}_1)-xH_1^\prime$ for some $H_1 ^\prime \in A_\tau$.  We also set $G_1 ^\prime = G_1+x^{s_1}Q \in A_{\tau}$, and thus clearly see from $\eqref{eq1}$ that $\phi ^\prime \circ \alpha ^\prime$ is of the required form:
\begin{align*}
\phi ^\prime \circ \alpha ^\prime &= (y_1+xF_1,\ldots,y_m+xF_m,z_1+xz_1G_1^\prime+x^{-t_1+(s_1+1)+1}H_1 ^\prime, \\
&\phantom{xxx} z_2+xz_2G_2+x^{-t_2+s_2+1}H_2, \ldots, z_n+xz_nG_n+x^{-t_n+s_n+1}H_n)
\end{align*}
\end{proof}

\begin{corollary}\label{IAreduce}
Let $\sigma \leq \tau \in \IN ^n$, and let $\alpha \in \IA _{m+n} ^\tau (R)$.  Then there exists $\beta \in \IA _{m+n} ^\sigma (R^{[m]})$ and $\phi \in \EA _n ^{\tau}(R^{[m]})$ such that $\alpha = \beta \circ \phi$.  Moreover, if $\tau - \sigma = (0,\ldots,0,\delta,0,\ldots,0)$ then $\phi$ can be taken to be elementary.
\end{corollary} 

\begin{theorem}\label{crucial}Let $\tau \in \IN ^n$, $\alpha \in \IA_{m+n}^\tau(R)$, and let $\phi \in \GA _n ^\tau (R^{[m]})$.  Then there exists $\tilde{\phi} \in \langle \phi, \EA_n ^\tau (R^{[m]}) \rangle$ such that $$\tilde{\phi} \circ \alpha \circ \phi \in \bigcap _{{\mathbf 0} \leq \sigma \leq \tau} \IA_{m+n}^\sigma (R)$$ 
In particular, $\tilde{\phi} \circ \alpha \circ \phi \in \IA _{m+n} (R)$; and if $\phi, \alpha \in \TA _{m+n}(S)$, then $\tilde{\phi} \circ \alpha \circ \phi \in \TA _{m+n}(S)$ as well.
\end{theorem}
\begin{proof}
By Lemma \ref{alphalemma},  $\phi ^{-1} \circ \alpha \circ \phi \in \IA _{m+n} ^\tau (R)$.  But then by Lemma \ref{strongIA}, there exists $\psi \in \EA _n ^\tau (R^{[m]}$) such that $\displaystyle \psi \circ (\phi ^{-1} \circ \alpha \circ \phi) \in \cap _{{\mathbf 0} \leq \sigma \leq \tau} \IA _{m+n} ^\sigma (R)$.  So we simply set $\tilde{\phi}=\psi \circ \phi ^{-1} \in \langle \phi, \EA _n ^\tau(R^{[m]}) \rangle$ to obtain the desired result.
\end{proof}



At this point, one could go ahead and directly prove Main Theorem 1.  However, it will be useful in proving Main Theorem 2 to have the stronger result of Theorem \ref{mt1strong} (which immediately implies Main Theorem 1).  To prove Theorem \ref{mt1strong}, we need to study $\GP_n(S^{[m]})$ and its relation with $\IA _{m+n}^\tau(R)$.
\begin{definition}\label{permdef}
Let $A$ be a connected, reduced ring.  Given $\rho \in \GP _n (S^{[m]})$, we can then write $\rho = (\lambda _{\sigma(1)} x^{r_{\sigma(1)}}z_{\sigma(1)}, \ldots, \lambda _{\sigma(n)} x^{r_{\sigma(n)}}z_{\sigma(n)})$ for some permutation $\sigma \in \mathfrak{S}_n$, $\lambda _i \in A^*$, and $r_i \in \IZ$.  If we are also given $\tau=(t_1,\ldots,t_n) \in \IN ^n$, we can define $\rho(\tau)=(t_{\sigma^{-1}(1)}+r_{1},\ldots,t_{\sigma^{-1}(n)}+r_{n})\in \IZ ^n$.  
\end{definition}
\begin{remark}
The condition that $A$ is  connected and reduced is essential to obtain $(A^{[m]}[x,x^{-1}])^*=\{\lambda x^r\ | \lambda \in A^*, r \in \IZ\}$, which is what allows us to write $\rho$ in the given form. 
\end{remark}

The definition of $\rho(\tau)$ is chosen precisely so that  the following lemma holds.
\begin{lemma}Let $A$ be a connected, reduced ring, let $\rho \in \GP_n(S^{[m]})$ and let $\tau \in \IN ^n$.  Then $\rho(A_{\tau})=A_{\rho(\tau)}$.
\end{lemma}

Recall that for $\tau = (t_1,\ldots, t_n) \in \IN ^n$, one obtains $\phi^\tau$ from $\phi$ via conjugation by $(x^{t_1}z_1,\ldots,x^{t_n}z_n) \in \GP _n (S^{[m]})$.  Then, recalling that for any subgroup $H \leq \GA_n(R^{[m]})$, $H^\tau = \{\phi ^\tau\ |\ \phi \in H\}$, we immediately see the following.  
\begin{lemma}
Let $H \leq \GA _{n}(R^{[m]})$, $\tau \in \IN ^n$, and $\rho \in \GP _n (S^{[m]})$.  Let $\phi \in \GA _n(S^{[m]})$.  Then $\phi \in H^\tau$ if and only if $\rho ^{-1} \circ \Phi \circ \rho \in H^{\rho(\tau)}$.
\end{lemma}

\begin{corollary}\label{IApermutation}
Let $A$ be a connected, reduced ring, let $\tau \in \IN ^n$, and let $\alpha \in \IA _{m+n}^\tau (R)$ and $\rho \in \GP _n (S^{[m]})$.  Then $\rho ^{-1} \circ \alpha \circ \rho \in \IA _{m+n} ^{\rho(\tau)} (R)$.
\end{corollary}

\begin{corollary}\label{earhotau}
Let $A$ be a connected, reduced ring, let $\tau \in \IN ^n$, let $\Phi \in \EA _n ^{\tau} (R^{[m]})$ be elementary, and let $\rho \in \GP _n (S^{[m]})$.  Then $\Phi \in \EA _n ^{\rho(\tau)}(R^{[m]})$ if and only if $\rho \circ \Phi \circ \rho ^{-1} \in \EA _n ^\tau (R^{[m]})$.
\end{corollary}

\begin{corollary}\label{rhopush}
Let $A$ be a connected, reduced ring, let $\phi \in \EA _n ^{\tau}(R^{[m]})$ and $\rho \in \GP _n (S^{[m]})$.  Then there exists $\phi ^\prime \in \EA _n ^{\rho(\tau)}(R^{[m]})$ such that $\phi \circ \rho = \rho \circ \phi ^\prime$.  Moreover, if $\phi$ is elementary, then so is $\phi ^\prime$.
\end{corollary}

We now have the necessary tools to prove Theorem \ref{mt1strong}, and the interested reader may skip ahead.  The rest of this section develops some more tools for use in the proof of Main Theorem 2.

\begin{lemma} \label{NoAlphaNoRho}
Let $A$ be a connected, reduced ring.  Let $\tau _0, \ldots, \tau _{q+1} \in \IN ^n$.  Let $\rho _i \in \GP _n (S^{[m]})$, $\alpha _i \in \IA _{m+n} ^{\tau _i}(R)$, and $\phi _i \in \EA _n ^{\tau _{i+1}}(R^{[m]})$ for $0 \leq i \leq q$.  Suppose also that $\rho _i (\tau _i) \leq \tau _{i+1}$ for each $ 0 \leq i \leq q$.  Then there exist  $\alpha ^\prime \in \IA _{m+n} ^{\tau _0}(R)$  and, for each $0 \leq i < q$, $\phi _i ^\prime \in \EA _n ^{(\rho _{i+1} \circ \cdots \circ \rho _{q})(\tau_{i+1})} (R^{[m]})$ such that 
$$\alpha _0 \circ \rho _0 \circ \phi _0 \circ \cdots \circ \alpha _ q \circ \rho _q \circ \phi _q = \alpha  ^\prime \circ (\rho _0 \circ \cdots \circ \rho _q)  \circ \phi _0 ^\prime \circ \phi _{1} ^\prime \circ \cdots \circ \phi _{q-1} ^\prime \circ \phi _q$$
\end{lemma}
\begin{proof}
We induct on $q$.  If $q=0$, the claim is trivial, so assume $q>0$.  So by the inductive hypothesis, we may assume
$$\alpha _1 \circ \rho _1 \circ \phi _1 \circ \cdots \circ \alpha _ q \circ \rho _q \circ \phi _q=
\alpha _1 ^\prime \circ \rho _1 ^\prime \circ \phi _1 ^\prime \circ \cdots \circ \phi _{q-1} ^\prime \circ \phi _q$$
for some $\alpha _1 ^\prime \in \IA _{m+n} ^{\tau _1}(R)$, $\phi _i ^\prime \in \EA _n ^{(\rho _{i+1} \circ \cdots \circ \rho _q)(\tau _{i+1})}(R^{[m]})$, and $\rho _1 ^\prime = \rho _1 \circ \cdots \circ \rho _q$.  Note that it now suffices to find $\alpha  ^\prime \in \IA _n ^{\tau _0}(R^{[m]})$ and $\phi _{0} ^\prime \in \EA _2 ^{\rho _1 ^\prime(\tau _1)}(R^{[m]})$ such that
$$\alpha _{0}  \circ \rho _{0} \circ \phi _{0} \circ \alpha _1 ^\prime \circ \rho _1 ^\prime  = \alpha  ^\prime \circ (\rho _0 \circ \rho _1 ^\prime) \circ \phi _{0} ^\prime$$

From Corollary \ref{alphapush}, there exists $\tilde{\alpha} \in \IA _{m+n} ^{\tau _1}(R)$ such that $ \phi_0 \circ \alpha _1 ^\prime =\tilde{\alpha} \circ \phi_0 $.  By Corollary \ref{IApermutation}, there exists $\alpha ^\pprime \in \IA _{m+n} ^{\rho _0 ^{-1} (\tau _1)}(R)$ such that  $\rho _0 \circ \tilde{\alpha} = \alpha ^\pprime \circ \rho _0$.  In addition, by Corollary \ref{IAreduce}, since $\tau _0  \leq \rho _0 ^{-1} (\tau _1)$, there exist $\beta \in \IA _{m+n} ^{\tau _0}(R)$ and $\tilde{\phi} \in \EA _n ^{\rho _0 ^{-1}(\tau _1)}(R^{[m]})$ such that $\alpha ^\pprime = \beta \circ \tilde{\phi}$.  Then we have
\begin{equation}
\alpha _{0}  \circ \rho _{0} \circ \phi _{0} \circ \alpha _1 ^\prime \circ \rho _1 ^\prime  = \alpha _0  \circ \rho _0 \circ \tilde{\alpha} \circ \phi _{0} \circ \rho _1 ^\prime 
=\alpha _0 \circ \alpha ^\pprime \circ \rho _0 \circ \phi _{0} \circ \rho _1 ^\prime 
=\alpha _0 \circ \beta \circ \tilde{\phi} \circ \rho _0 \circ \phi _{0} \circ \rho _1 ^\prime 
\label{eqAlphaRhoPhiRho}
\end{equation}

Now by Corollary \ref{rhopush}, there exists $\phi ^\prime \in \EA _2 ^{\tau _1}(R^{[m]})$ such that $\tilde{\phi} \circ \rho _0 = \rho _0 \circ \phi ^\prime$.  Also, there exist $\phi _0 ^\prime \in \EA _2 ^{\rho _1 ^\prime (\tau _1)}(R)$ such that $ (\phi ^\prime \circ \phi _0) \circ \rho _1 ^\prime  = \rho _1 ^\prime  \circ \phi _0 ^\prime$.  Then from \eqref{eqAlphaRhoPhiRho}, we obtain 

\begin{equation*}
\alpha _{0}  \circ \beta \circ \tilde{\phi} \circ \rho _{0} \circ \phi _{0} \circ \alpha _1 ^\prime  \circ \rho _1 ^\prime 
=\alpha _0 \circ \beta \circ \rho _0 \circ (\phi ^\prime \circ \phi _0 ) \circ \rho _1 ^\prime
=\alpha _0 \circ \beta \circ \rho _0 \circ \rho _1 ^\prime \circ \phi _{0} ^\prime 
\end{equation*}
Now we simply set $\alpha ^\prime = \alpha _0 \circ \beta \in \IA _{m+n} ^{\tau _0}(R)$  to achieve the desired result.
\end{proof}

In fact, this same proof gives the following, noting that the hypothesis $\tau _2 - \rho_1(\tau _1) = \delta e_k$ is what implies that the resulting $\Phi$ is elementary:
\begin{corollary}\label{alpharhoalpharho}
Let $A$ be a connected, reduced ring.  Suppose $\tau _1, \tau _2 \in \IN ^n$ , $\alpha _1 \in \IA _{m+n} ^{\tau _1}(R)$, $\alpha _2 \in \IA _{m+n} ^{\tau _2}(R)$, and $\rho _1, \rho _2 \in \GP _2 (S^{[m]})$.  If $\tau _2 - \rho _1(\tau _1) = \delta e_k$ for some $1 \leq k \leq n$ and $\delta \in \IN$, then there exist $\alpha ^\prime \in \IA _{m+n} ^{\tau _1}$, $\rho ^\prime \in \GP _2 (S^{[m]})$, and elementary $\Phi \in \EA _n ^{\rho _2(\tau _2)}(R^{[m]})$ such that
$$ \alpha _1 \circ \rho _1 \circ \alpha _2 \circ \rho _2 = \alpha ^\prime \circ \rho ^\prime \circ \Phi$$
\end{corollary}

We alert the reader to the fact that the next lemma is true only for $n=2$.
\begin{lemma}\label{PhiEAtau}
Assume $A$ is an integral domain of characteristic zero.  Let $\sigma \leq \tau=(t_1,t_2) \in \IN ^2$, let $\Phi \in \EA _2 (S^{[m]})$ be elementary, and let $\omega \in \GA _{m+2}(S)$ such that $\omega (x^{t_1}z_1), \omega(x^{t_2}z_2) \in R^{[m+2]} \setminus xR^{[m+2]}$.  If $(\Phi \circ \omega)(A_{\sigma}) \subset R^{[m+2]}$, then $\Phi \in \EA _2 ^{\tau}(R^{[m]})$.
\end{lemma}
\begin{proof} 
Without loss of generality, assume that $\Phi$ is elementary in $z_1$.  Then setting $\sigma = (s_1,s_2)$, it is clear that $s_2=t_2$.  Write $\Phi = (z_1 + x^{-r}P(x^{t_2}z_2),z_2)$ for some $P \in A_\tau[\hat{z}_1] \setminus xA_\tau[\hat{z}_1]$.  It suffices to show that $r \leq t_1$.  We compute
$$(\Phi \circ \omega)(x^{t_1}z_1)=\omega(x^{t_1}z_1)+x^{t_1-r}P(\omega(x^{t_2}z_2))$$
Since $\sigma \leq \tau$ and $(\Phi \circ \omega)(A_\sigma) \in R^{[m+2]}$, we must have $(\Phi \circ \omega)(x^{t_1}z_1) \in R^{[m+2]}$.  But $\omega (x^{t_1}z_1) \in R^{[m+2]}$ by assumption, so we then have $x^{t_1-r}P(\omega(x^{t_2}z_2)) \in R^{[m+2]}$.  Since $P \notin (x)$ and $\omega (x^{t_2}z_2) \in R^{[m+2]} \setminus xR^{[m+2]}$, we thus must have $r \leq t_1$ as required.
\end{proof}

We remark that if $n=3$, one may find $P \notin (x)$ and $\omega$ with $\omega(x^{t_2}z_2), \omega (x^{t_3}z_3) \in R^{[m+2]} \setminus xR^{[m+2]}$, but $P(\omega(x^{t_2}z_2),\omega(x^{t_3}z_3)) \in xR^{[m+2]}$.  For example, set $\tau = (0,1,2)$, $\omega = (z_1,z_2-\frac{yz_1}{x}, z_3-\frac{2z_2z_1}{x}-\frac{yz_1^2}{x^2}) \in \GA _3(\IC[x,x^{-1}][y])$, and let $P=y(x^{2}z_3)+(xz_2)^2 \in A_{\tau}[\hat{z}_1] \setminus xA_{\tau}[\hat{z}_1]$.  Then one easily checks that $P(\omega(xz_2,x^2z_3))=x^2(yz_3+z_2^2) \in (x^2)\IC[x][y]^{[3]}$.  This type of behavior is a crucial difficulty in extending our result from the $n=2$ case to $n \geq 3$.  Additionally, if we relax our hypotheses on the ring $A$, we create similar difficulties.

\begin{lemma}\label{acnonzero}
Let $\tau = (t_1,t_2) \in \IN ^2$.  Let $\Phi \in \EA _2 ^\tau (R^{[m]})$ be elementary, and let $\beta = (az_1+bx^{t_2-t_1}z_2,dz_2+cx^{t_1-t_2}z_1) \in \GL _2 ^\tau (R^{[m]})$.  
\begin{enumerate}
\item If $\Phi$ is elementary in $z_1$ and either $c=0$ or $d=0$, then there exists $\rho \in \GP _2 ^\tau (R^{[m]})$ and elementary $\Phi ^\prime \in \EA _2 ^\tau (R^{[m]})$ such that $ \Phi \circ \beta = \rho \circ \Phi ^\prime$.
\item If $\Phi$ is elementary in $z_2$ and either $a=0$ or $b=0$, then there exists $\rho \in \GP _2 ^\tau (R^{[m]})$ and elementary $\Phi ^\prime \in \EA _2 ^\tau (R^{[m]})$ such that $ \Phi \circ \beta = \rho \circ \Phi ^\prime$.
\end{enumerate}
\end{lemma}
\begin{proof}
Suppose $\Phi$ is elementary in $z_1$ and write $\Phi = (z_1+ x^{-t_1}P(x^{t_2}z_2),z_2)$.  First, suppose $c=0$, so
$$\Phi \circ \beta = (az_1+bx^{t_2-t_1}z_2+x^{-t_1}P(dx^{t_2}z_2),dz_2) = (az_1,dz_2) \circ (z_1+\frac{1}{a}x^{-t_1}(bx^{t_2}z_2+P(dx^{t_2}z_2)),z_2)$$
If instead $d=0$, then
$$\Phi \circ \beta = (az_1+bx^{t_2-t_1}z_2+x^{-t_1}P(cx^{t_1}z_1),cx^{t_1-t_2}z_1) = (bx^{t_2-t_1}z_2,cx^{t_1-t_2}z_1) \circ (z_1,z_2+\frac{1}{b}x^{-t_2}(ax^{t_1}z_1+P(cx^{t_1}z_1)),z_2)$$
These are both precisely in the desired form.  The case where $\Phi$ is elementary in $z_2$ follows similarly.
\end{proof}

We conclude with two technical lemmas.

\begin{lemma}\label{technicallemma}
Suppose $A$ is a connected, reduced ring.  Let $\tau _0, \ldots, \tau _q \in \IN ^n$, $\Phi _0, \ldots, \Phi _q \in \EA _n(S^{[m]})$ be elementaries, $\alpha _i \in \IA_{m+n} ^{\tau _i}(R)$, and $\rho _i \in \GP _n (S^{[m]})$.  Set
$$\omega _i = \alpha _i \circ \rho _i \circ \Phi _i \circ \cdots \circ \alpha _q \circ \rho _q \circ \Phi _q$$ 
Also set $\Phi _i ^\prime = \rho _i \circ \Phi _i \circ \rho _i ^{-1}$.  Then the following conditions are equivalent:
\begin{enumerate}
\item Each $\tau _i \in \IN ^n$ is minimal such that $\omega _i (A_{\tau _i}) \subset R^{[m+n]}$
\item Each $\tau _i \in \IN ^n$ is minimal such that $(\Phi _i \circ \omega _{i+1})(A_{\rho _i (\tau _i)}) \subset R^{[m+n]}$.
\item Each $\tau _i \in \IN ^n$ is minimal such that $(\rho _i \circ \Phi _i \circ \omega _{i+1})(A_{\tau _i}) \subset R^{[m+n]}$.
\item Each $\tau _i \in \IN ^n$ is minimal such that $(\Phi _i ^\prime \circ \rho _i \circ \omega _{i+1})(A_{\tau _i}) \subset R^{[m+n]}$.
\end{enumerate}

Moreover, if the above are satisfied, then, writing $\tau _i = (t_{i,1},\ldots, t_{i,n})$,
\begin{enumerate}
\item If $\Phi _i ^\prime $ is elementary in $z_j$, then $(\rho _i \circ \omega _{i+1})(x^{t_{i,k}}z_k) \in R^{[m+n]}\setminus xR^{[m+n]}$ for all $k \neq j$.
\item If $\Phi _i$ is elementary in $z_j$, then $\rho _i (\tau _i)-\tau _{i+1} = \delta _i e_j$  for some $\delta _i \in \IZ$ (recall $e_j=(0,\ldots,0,1,0,\ldots,0)$).
\end{enumerate}
\end{lemma}
\begin{proof}
The equivalence of (2) and (3) is immediate from the fact that $\rho _i (A_{\tau _i}) = A_{\rho _i(\tau _i)}$.  Since $\alpha _i \in \IA_{m+n} ^{\tau _i} (R) \subset \GA _{m+n} ^{\tau _i}(R)$, we have $\alpha _i (A_{\tau _i})= A_{\tau _i}$ and thus $\omega _i (A_{\tau _i})=(\rho _i \circ \Phi _i \circ \omega _{i+1})(A_{\tau _i})$, giving the equivalence of (1) and (3).  The equivalence of (3) and (4) follows immediately from the definition of $\Phi _i ^\prime$.

Suppose now that the four conditions are satisfied.  Suppose also that $\Phi _i ^\prime $ is elementary in $z_j$, so that $\Phi _i ^\prime(z_k)=z_k$ for $k \neq j$.  Then (4) immediately implies $(\rho _i \circ \omega _{i+1})(x^{t_{i,k}}z_k)=(\Phi _i ^\prime \circ \rho _i \circ \omega _{i+1})(x^{t_{i,k}}z_k) \in R^{[m+n]} \setminus xR^{[m+n]}$.  
Now suppose (perhaps instead) that $\Phi _i$ is elementary in $z_j$.  Then $(\Phi _i \circ \omega _{i+1})(x^sz_k) = \omega _{i+1}(x^sz_k)$ for $k \neq j$.  The minimal $s$ such that this lies in $R^{[m+n]}$ is  precisely $t_{i+1,k}$, so we see from (2) that $\rho _i (\tau _i)=\tau _{i+1}+\delta _ie_j$ for some $\delta _i \in \IZ$. 
\end{proof}

We make the following definition to aid in the proof of the Lemma \ref{onlyOnePhi}. 

\begin{definition}
Let $\tau=(t_1,\ldots,t_n) \in \IN$. Note that as in \eqref{shortexact}, we have $\GA_n (A^{[m]}) \leq \GA_n (R^{[m]})$.  Then, we can consider $\EA_n(A^{[m]}) \leq \GA_n(R^{[m]})$, and define $$\EA_n ^\tau (A^{[m]}) := \{ (x^{-t_1}z_1,\ldots,x^{-t_n}z_n) \circ \phi \circ (x^{t_1}z_1,\ldots,x^{t_n}z_n)\ \big |\ \phi \in \EA _n ^\tau (A^{[m]})\} \leq \GA_n ^\tau (R^{[m]})$$
Given $\Phi \in \EA _n ^\tau (R^{[m]})$, we will denote its image under the natural quotient as $\bar{\Phi} \in \EA _n ^\tau (A^{[m]})$.

We can define other subgroups such as $\GL _2 ^\tau (A^{[m]})$ in a similar way.
\end{definition}

\begin{lemma}\label{eabar}
Let $\tau \in \IN$, and let $\Phi _1, \ldots, \Phi _q \in \EA _n ^\tau (R^{[m]})$.  Then there exist $\alpha \in \IA _n ^\tau (R^{[m]})$ and $\tilde{\Phi}_1, \ldots, \tilde{\Phi} _q \in \EA _n ^\tau(A^{[m]})$ such that $\Phi _1 \circ \cdots \circ \Phi _q = \alpha \circ \tilde{\Phi} _1 \circ \cdots \circ \tilde{\Phi} _q$.
\end{lemma}
\begin{proof}
The key observation is that if $\Phi \in \EA_n ^\tau (R^{[m]})$, then $\Phi \circ \overline{\Phi} ^{-1} \in \IA _n ^\tau (R^{[m]})$.  Then Corollary \ref{alphapush} and a quick induction suffice to prove the lemma.
\end{proof}

\begin{lemma} \label{onlyOnePhi}
Suppose $A$ is an integral domain.  Let $\tau =(t_1,t_2)\in \IN ^2$, let $\Phi _1, \ldots, \Phi _q \in \EA _2 ^\tau (R^{[m]})$ be elementaries, and let $\omega \in \GA_2(S^{[m]})$ .
Assume that
\begin{enumerate}
\item Either $\omega (x^{t_1}z_1) \in xR^{[m+2]}$ and $\omega(x^{t_2}z_2) \in R^{[m+2]} \setminus xR^{[m+2]}$, or $\omega(x^{t_2}z_2) \in xR^{[m+2]}$ and $\omega(x^{t_1}z_1) \in R^{[m+2]}\setminus xR^{[m+2]}$.
\item Setting $\omega _i = \Phi _i \circ \cdots \circ \Phi _q \circ \omega$, $\omega _i (x^{t_1}z_1), \omega _i (x^{t_2}z_2) \in R^{[m+2]} \setminus xR^{[m+2]}$ for $1<i \leq q$
\item $\omega _1 (x^{t_1}z_1) \in xR^{[m+2]}$
\end{enumerate}
Then either all $\Phi _i$ are elementary in the same variable, or there exists $\alpha \in \IA _{2} ^{\tau}(R^{[m]})$, $\rho \in \GP _2 ^\tau (R^{[m]})$ and elementary $\Phi \in \EA _2 ^\tau (R^{[m]}) \cap \GL _2 ^\tau(R^{[m]})$ such that $\Phi _1 \circ \cdots \circ \Phi _q =  \alpha \circ \rho \circ \Phi$.  
\end{lemma}

\begin{proof}
Note that it suffices to assume that $\Phi _i$ and $\Phi _{i+1}$ are not elementary in the same variable for each $i$, and we may assume $q \geq 2$.  Moreover, by Lemma \ref{eabar}, we may write $\Phi _1 \circ \cdots \circ \Phi _q = \alpha \circ \tilde{\Phi} _1 \circ \cdots \circ \tilde{\Phi} _q$, for some $\tilde{\Phi} _i \in \EA _2 ^\tau (A^{[m]})$ and $\alpha \in \IA_2 ^\tau (R^{[m]})$.  So without loss of generality, it suffices to assume $\alpha = \id$ and each $\Phi _i \in \EA _2 ^\tau (A^{[m]})$.  We thus write, for each $1 \leq i \leq q$, (assuming $\Phi _i$ is elementary in $z_1$) $\Phi _i = (z_1+x^{-t_1}P_i(x^{t_2}z_2),z_2)$ for some $P_i \in A^{[m]}[x^{t_2}z_2] \subset A _\tau[\hat{z}_1]$.

By assumption 2, for $i>1$ we can write $\omega _i = (x^{-t_1}F_i, x^{-t_2}G_i)$ for some $F_i,G_i \in R^{[m+2]} \setminus xR^{[m+2]}$.  Given $Q \in R^{[m]}[z_1,z_2]$, we will use $\overline{Q}$ to denote its image (under the quotient map modulo $x$) in $A^{[m]}[z_1,z_2]$.  Thus, we can interpret assumption 2 as $\overline{F_i} \neq 0$ and $\overline{G_i} \neq 0$ for $1<i\leq q$.  We inductively show the following claim:

\begin{claim}\label{degreeFGclaim}
For each $i>1$, there exist $\Phi _{i} ^\prime, \ldots, \Phi _q ^\prime \in \EA _2 ^\tau(A^{[m]})$ and $\rho \in \GP _2 ^\tau (R^{[m]})$ such that
\begin{enumerate}
\item $\Phi _i \circ \cdots \circ \Phi _q = \rho \circ \Phi _{i}^\prime \circ \cdots \circ \Phi _q ^\prime$
\item Letting $\omega _i ^\prime = \Phi _i ^\prime \circ \cdots \circ \Phi _q ^\prime \circ \omega$ and setting $\omega _i ^\prime(x^{t_1}z_1) = F_i ^\prime$ and $\omega _i ^\prime (x^{t_2}z_2) = G_i ^\prime$, then
\begin{enumerate}
\item If $\Phi _i ^\prime$ is nonlinear and elementary in $z_1$, then $\deg \overline{F_i ^\prime} > \deg \overline{G_i ^\prime}$
\item If $\Phi _i ^\prime$ is nonlinear and elementary in $z_2$, then $\deg \overline{F_i ^\prime} < \deg \overline{G_i ^\prime}$
\end{enumerate}
\end{enumerate}
\end{claim}
Let us first see how this claim completes the lemma.  Applying the claim for $i=2$, we obtain  
\begin{align*}
\Phi _1 \circ \cdots \circ \Phi _q &= \Phi _1 \circ \rho \circ \Phi _2 ^\prime \circ \cdots \circ \Phi _q ^\prime = \rho \circ \Phi _1 ^\prime \circ \Phi _2 ^\prime \circ \cdots \circ \Phi _q ^\prime
\end{align*}
with the final equality coming from Corollary \ref{rhopush}.  Here, each $\Phi _i ^{\prime} \in\EA _2 ^\tau (A^{[m]})$.  

Note that it suffices to assume $\rho = \id$.  Without loss of generality, assume $\Phi _1 ^\prime$ is elementary in $z_1$ and $\Phi _2 ^\prime$ is elementary in $z_2$.  Then we compute
\begin{align*}
\omega _1 (x^{t_1}z_1) &= (\Phi _1 ^\prime \circ \omega _2 ^\prime)(x^{t_1}z_1) = \omega _2 ^\prime (x^{t_1}z_1+P_1(x^{t_2}z_2)) = F_2 ^\prime + P_1(G_2 ^\prime)
\end{align*}

But assumption 3 implies that $\overline{F_2 ^\prime}+P_1(\overline{G_2 ^\prime}) \equiv 0$, and thus $\deg \overline{F_2 ^\prime} = (\deg P_1 )(\deg \overline{G_2 ^\prime})$.  Since the claim yields that if $\Phi _2 ^\prime$ is nonlinear, then $\deg \overline{F_2} ^\prime < \deg \overline{G_2 ^\prime}$, we must have that $\Phi _2 ^\prime$ is linear.  Let $b\geq 2$ be minimal such that $\Phi _{b+1} ^\prime$ is non-linear (and thus $\Phi _2, \ldots, \Phi _b$ are all linear).  We will derive a contradiction, showing no such $b$ exists, in which case $\Phi _2 ^\prime, \ldots, \Phi _q ^\prime$ are all linear.

Set $\beta = \Phi _2 ^\prime \circ \cdots \circ \Phi _b ^\prime = (\beta _{11}z_1+\beta_{12}x^{t_2-t_1}z_2,\beta_{22}z_2+\beta_{21}x^{t_1-t_2}z_1) \in \GL _2 ^\tau (A^{[m]})$ (for some $\beta _{ij} \in A^{[m]}$).  Note that by Lemma \ref{acnonzero}, if $\beta _{21}=0$ or $\beta _{22}=0$, we may (absorbing a resulting permutation into $\rho$) replace $\beta$ by the identity.  Thus, we assume without loss of generality that  $\beta _{21} \neq 0$ and $\beta _{22} \neq 0$.

But then we have 
\begin{align}
F_2^\prime &= \omega _2 ^\prime (x^{t_1}z_1) = \left( \beta \circ  \omega _{b+1} ^\prime\right)(x^{t_1}z_1) = \omega _{b+1} ^\prime (\beta _{11}x^{t_1}z_1+\beta _{12} x^{t_2}z_2) = \beta _{11}F_{b+1} ^\prime + \beta _{12} G_{b+1} ^\prime \notag \\
G_2 ^\prime &= \omega _2 ^\prime (x^{t_2}z_2) = \left( \beta \circ  \omega _{b+1} ^\prime\right)(x^{t_2}z_2) = \omega _{b+1} ^\prime (\beta _{22}x^{t_2}z_2+\beta _{21} x^{t_1}z_1) = \beta _{21}F_{b+1} ^\prime + \beta _{22} G_{b+1} ^\prime \label{F2G2}
\end{align}
Thus, since $\overline{F_2^\prime}+P_1(\overline{G_2^\prime})=0$, we have from \eqref{F2G2}
\begin{equation}
\beta _{11}\overline{F_{b+1}^\prime}+\beta _{12}\overline{G_{b+1}^\prime}+P_1(\beta _{21}\overline{F_{b+1}^\prime}+\beta _{22}\overline{G_{b+1}^\prime}) = 0 \label{FGP1}
\end{equation}

Since $\Phi _{b+1} ^\prime$ is nonlinear, we must have (by the claim) $\deg {\overline{F_{b+1} ^\prime}} \neq \deg \overline{G_{b+1}^\prime}$; then since $\beta_{21} \neq 0$ and $\beta _{22} \neq 0$, from \eqref{FGP1} we see
$$\max \left\{\deg \overline{F_{b+1}^\prime}, \deg \overline{G_{b+1}^\prime}\right\} \geq \deg\left(\beta _{11}\overline{F_{b+1}^\prime}+\beta _{12}\overline{G_{b+1}^\prime}\right) = (\deg P_1) \max \left\{\deg \overline{F_{b+1}^\prime}, \deg \overline{G_{b+1}^\prime} \right\}$$

Thus we must have $\deg P_1=1$.  Let $P_1(z)=\mu z$ for some $\mu \in A^{[m]}$.  Then (again since $\deg \overline{F_{b+1}^\prime} \neq \deg \overline{G_{b+1} ^\prime}$) we see $\beta _{11}+\mu \beta _{21}=0$ and $\beta _{12}+\mu\beta _{22}=0$; however, this implies $\det \beta =\det \begin{pmatrix} \beta _{11} & \beta _{12} \\ \beta _{21} & \beta _{22}\end{pmatrix}=0$, contradicting $\beta \in \GL _2 ^\tau (R^{[m]})$.  

So we now have that $\Phi _2 ^\prime, \ldots, \Phi _q ^\prime$ are all linear.  We will continue to write $\beta = \Phi _2 ^\prime \circ \cdots \circ \Phi _q ^\prime \in \overline{\GL _2 ^\tau} (R^{[m]})$.  Now write $\omega (x^{t_{i,1}}z_1)=F$ and $\omega (x^{t_{i,2}}z_2)=G$.  We have $F,G \in R^{[m+2]}$, but by assumption 1, either $\overline{F}=0$ or $\overline{G}=0$.  Then we compute
$$\omega _1 (x^{t_1}z_1) = \left( \Phi _1 \circ \beta \circ  \omega \right)(x^{t_1}z_1) = \omega(\beta _{11}x^{t_1}z_1+\beta _{12} x^{t_2}z_2+P_1(\beta_{21}x^{t_1}z_2+\beta_{22}x^{t_2}z_2)) = \beta _{11}F + \beta _{12} G+P_1(\beta_{21}F+\beta_{22}G)$$
But since $\overline{F}=0$ or $\overline{G}=0$, we clearly must have $P_1$ is linear.  We may then write $\Phi _1 ^\prime \circ \beta = (az_1+bx^{t_2-t_1}z_2,cx^{t_1-t_2}z_1+dz_2)$ (for some $a,b,c,d \in A^{[m]}$), and compute $\omega _1 (x^{t_1}z_1)=aF+bG$.  Since assumption 3 implies $a\overline{F}+b\overline{G}=0$, and either $\overline{F}=0$ or $\overline{G}=0$ (but not both), we must have $a=0$ or $b=0$.  
 Then from Lemma \ref{acnonzero}, we have
$$\Phi _1 ^\prime \circ \cdots \circ \Phi _q ^\prime = \beta = \rho \circ \Phi$$
for some $\rho \in \GL _2 ^\tau (R^{[m]})$ and elementary $\Phi \in \EA _2 ^\tau (R^{[m]}) \cap \GL _2 ^\tau (R^{[m]})$, as required.

We thus are reduced to proving Claim \ref{degreeFGclaim}.
\end{proof}
\begin{proof}[Proof of Claim \ref{degreeFGclaim}]
The proof is induction on $i$.  First, suppose $i=q$.  We set $\Phi _q ^\prime = \Phi _q$, and without loss of generality, assume $\Phi _q ^\prime$ is elementary in $z_1$ (the case where it is elementary in $z_2$ follows similarly).  Write $\omega = (x^{-t_1}F, x^{-t_2}G)$ for some $F,G \in R^{[m+2]}$.  Note that our assumption that $\Phi _q$ is elementary in $z_1$ (along with assumption 2) forces $\overline{F}=0$ and $\overline{G} \neq 0$.   Then
$$\omega _q ^\prime = \Phi _q ^\prime \circ \omega= (x^{-t_1}(F+P_q(G)),x^{-t_2}G)$$
Since $F \in (x)$, we thus have $\overline{F_q ^\prime} = P_q(\overline{G})$ and $\overline{G_q ^\prime}=\overline{G}$.  Then if $\Phi _q ^\prime$ is non-linear, $\deg P_q > 1$, and we have $\deg \overline{F_q ^\prime }=(\deg P_q)(\deg \overline{G}) > \deg \overline{G_q ^\prime}$ as required.

Now suppose $i<q$ with $\Phi _i$ non-linear.  By the induction hypothesis, we may replace $\Phi _j$ with $\Phi _j ^\prime$  for $j>i$ with the desired properties (using Corollary \ref{rhopush} to push any resulting permutation to the left).  Let $j>i$ be minimal such that $\Phi _j ^\prime$ is also non-linear.  Let $\beta = \Phi _{i+1} ^\prime \circ \cdots \circ \Phi _{j-1} ^\prime = (\beta _{11}z_1+\beta_{12}x^{t_1-t_2}z_2,\beta_{22}z_2+\beta_{21}x^{t_1-t_2}z_2) \in\GL _2 ^\tau(A^{[m]})$ for some $\beta _{i,j} \in A^{[m]}$.  Without loss of generality, assume $\Phi _i$ is elementary in $z_1$.  Then by Lemma \ref{acnonzero}, we may assume $\beta _{21} \neq 0$ and $\beta _{22} \neq 0$ by factoring through a permutation.  We then compute 
\begin{align*}
\omega _i &= \Phi _i \circ \beta \circ \omega _{j} ^\prime \\
&= \Phi _i \circ \left(x^{-t_1}(\beta _{11}F_j ^\prime + \beta _{12}G_j ^\prime), x^{-t_2}(\beta _{21}F_j ^\prime + \beta _{22}G_j ^\prime) \right) \\
&= \left(x^{-t_1}(\beta _{11}F_j ^\prime + \beta _{12}G_j ^\prime + P_i(\beta _{21}F_j+\beta _{22}G_j)), x^{-t_2}(\beta _{21}F_j ^\prime + \beta _{22}G_j ^\prime) \right) 
\end{align*}
Then clearly we have
\begin{align*}
F_i &= \beta _{11}F_j ^\prime + \beta _{12}G_j ^\prime + P_i(\beta _{21}F_j+\beta _{22}G_j) \\
G_i &= \beta _{21}F_j ^\prime + \beta _{22}G_j ^\prime
\end{align*}
Let $d=\max\{\deg \overline{F_j ^\prime}, \deg \overline{G_j ^\prime} \}$.  Then since $\deg \overline{G_j ^\prime} \neq \deg \overline{F_j ^\prime}$ by the inductive hypothesis, and since $\beta _{21} \neq 0$ and $\beta _{22} \neq 0$, $\deg \overline{G_i} = d$ and $\deg \overline{F_i} = (\deg P_i) d > d$ since $P_i$ is non-linear.  Thus we may now take $\Phi _i ^\prime = \Phi _i$ to complete the proof.
\end{proof}

\subsection{Main Theorems}\label{maintheorems}
We can now state and prove our main theorems.
\renametheorem{Main Theorem 1}
\begin{namedtheorem}
Let $\tau _0 \geq \cdots \geq \tau _q \in \IN ^n$.  For $0 \leq i \leq q$, let $\Phi _i \in \GA _n ^{\tau _i }(R^{[m]})$ and $\alpha _i \in \IA_{m+n} ^{\tau _i}(R)$.  Set $$\psi=\alpha _0 \circ \Phi _0 \circ \cdots \circ \alpha _q \circ \Phi _q$$
Then $(\psi (y_1), \ldots, \psi (y_m) )$ is a partial system of coordinates over $R$.  Moreover, if $A$ is a regular domain, and $\alpha _i, \Phi _i \in \TA_{m+n}(S)$ for $0 \leq i \leq q$, then $(\psi  (y_1), \ldots, \psi  (y_m))$ can be extended to a stably tame automorphism of $R^{[m+n]}$.
\end{namedtheorem}

We are now ready to prove the following result, which immediately implies Main Theorem 1.  The inclusion of the permutation maps $\rho _i$ is not necessary to achieve Main Theorem 1, but will help us in our proof of Main Theorem 2.  Note that if we assume each $\rho _i$ is of the form in Definition \ref{permdef}, then we may drop the assumption ``$A$ is a connected, reduced ring''.  In particular, we do not need to assume $A$ is connected and reduced in Main Theorem 1, since we set $\rho _i = \id$ for each $i$ to obtain it from Theorem \ref{mt1strong}. 

\begin{theorem}\label{mt1strong}
Let $A$ be a connected, reduced ring, and let $\tau _0, \ldots, \tau _q \in \IN ^n$.  Let $\rho _i \in \GP _n (S^{[m]})$,  $\alpha _i \in \IA _{m+n} ^{\tau _i}(R)$, and $\Phi _i \in \GA _n ^{\rho _i (\tau _i)}(R^{[m]})$ for each $0 \leq i \leq q$.  Set 
$$\psi _i = \alpha _0 \circ \rho _0  \circ \Phi _0 \circ  \cdots \circ \alpha _i \circ \rho _i  \circ \Phi _i$$ 
Suppose $\rho _i (\tau _i) \geq \tau _{i+1}$ for each $0 \leq i \leq q$.  Then for each $0 \leq i \leq q$, there exists $\theta _i \in \IA_{m+n} ^{\tau _{i+1}}(R) \cap \IA_{m+n}(R)$ with $\theta _i (y_j)=\psi _i(y_j)$ for each $1 \leq j \leq m$.  Moreover, if $\alpha _k, \Phi _k \in \TA_{m+n}(S)$ for $0 \leq k \leq i$, then $\theta _i$ is stably tame.
\end{theorem}
\begin{proof}
The proof is by induction on $i$.  Note that we may use a trivial base case of $i=-1$ and $\theta _{-1}=\id$.
So we suppose $i \geq 0$.
By the induction hypothesis we have $\theta _{i-1} \in  \IA _{m+n} ^{\tau _i}(R)$.  Thus, $(\theta _{i-1} \circ \alpha _i) \in \IA_{m+n} ^{\tau _i}(R)$, and by Corollary \ref{IApermutation}, $\rho _i ^{-1} \circ (\theta _{i-1} \circ \alpha _i) \circ \rho _i \in \IA _{m+n} ^{\rho _i (\tau _i)}(R)$.  Since $\Phi _i \in \GA _n ^{\rho _i (\tau _{i})}(R)$ and $\tau _{i+1} \leq \rho(\tau _i)$, we can apply Theorem \ref{crucial} to obtain $\tilde{\Phi} \in \GA _n ^{\rho _i (\tau _i)} (R^{[m]})$ such that 
$$\theta _i := \tilde{\Phi}  \circ (\rho _i ^{-1} \circ \theta _{i-1} \circ \alpha _i \circ \rho _i) \circ \Phi _i \in \IA_{m+n} ^{\tau _{i+1}}(R) \cap \IA _{m+n}(R)$$  
Noting that $\tilde{\Phi}, \rho _i \in \GA _n (S^{[m]})$ and thus fix each $y_j$, and by the inductive hypothesis $\theta _{i-1}(y_j)=\psi_{i-1}(y_j)$ we have 
\begin{align*}
\theta _i (y_j) &= (\tilde{\Phi} \circ \rho _i ^{-1} \circ \theta _{i-1} \circ \alpha _i \circ \rho _i \circ \Phi _i)(y_j) \\
&=(\theta _{i-1} \circ \alpha _i \circ \rho _i \circ \Phi _i)(y_j) \\
&=(\psi _{i-1} \circ \alpha _{i} \circ \rho _i  \circ \Phi _i)(y_j) \\
&= \psi_i(y_j)
\end{align*} for each $1\leq j \leq m$.  Moreover, if $\alpha _0, \Phi _0, \ldots, \alpha _i, \Phi _i \in \TA _{m+n}(S)$, then the inductive hypothesis along with Theorem \ref{crucial} guarantee $\theta _i \in \TA_{m+n}(S)$ as well.  Noting that since $\theta _i \in \IA_{m+n}(R)$ we have $\theta _i \equiv \id \pmod{x}$, the stable tameness assertion follows immediately from the following result of Berson, van den Essen, and Wright:
\begin{theorem}[\cite{BEW}, Theorem 4.5]Let $A$ be a regular domain, and let $\phi \in \GA_n(R)$ with $J\phi=1$.  If $\phi \in \TA_n(S)$ and $\bar{\phi} \in \EA_n(R/xR)$, then $\phi$ is stably tame.
\end{theorem}
\end{proof}


\renametheorem{Main Theorem 2}
\begin{namedtheorem}
Suppose $A$ is an integral domain of characteristic zero.  Let $\Phi _0, \ldots, \Phi _q \in \EA _2 (S^{[m]})$ be elementaries.  Let $\alpha _i \in \GA_{m+2}(S)$ and $\rho _i \in \GP _2(S^{[m]})$ for each $0 \leq i \leq q$.  Set
$$\omega _i = \alpha _i \circ \rho _i \circ \Phi _i \circ \cdots \circ \alpha _q \circ \rho _q \circ \Phi _q$$
and define $\tau _i \in \IN ^2$ to be minimal such that $\omega _{i}(A_{\tau _i}) \subset R^{[m+2]}$ for $0 \leq i \leq q$.
If $\alpha _i \in \IA_{m+2} ^{\tau _i}(R)$ for each $0 \leq i \leq q$, then there exists $\theta \in \IA_{m+2}(R)$ such that $\theta(y_j)=\omega _0 (y_j)$.
\end{namedtheorem}

The theorem follows from following claim, which allows us to apply Theorem \ref{mt1strong} to $\omega _0$.  By convention, we will let $\tau _{q+1}=0$.

\begin{claim}\label{inductionclaim}
For each $a \leq q$, there exist the following:
\begin{enumerate}
\item A sequence $\tilde{\tau}_a, \ldots , \tilde{\tau} _q \in \IN ^2$
\item $\tilde{\rho }_a, \ldots, \tilde{\rho} _q \in \GP _2 (S^{[m]})$
\item For each $a \leq i \leq q$, $\tilde{\alpha} _i \in \IA _{m+2} ^{\tilde{\tau} _i}(R)$ and $\Phi _i \in \EA _2 ^{\tilde{\rho} _i (\tilde{\tau} _i)} (R^{[m]})$
\end{enumerate}
such that
\begin{enumerate}
\item $\tilde{\rho}_i (\tilde{\tau} _i) \geq \tilde{\tau} _{i+1}$
\item Setting $\tilde{\omega _i} = \tilde{\alpha} _i \circ \tilde{\rho} _i \circ \tilde{\Phi} _i \circ \cdots \circ \tilde{\alpha} _q \circ \tilde{\rho} _q \circ \tilde{\Phi} _q $, each $\tilde{\tau}_i$ is minimal such that $\tilde{\omega} _{i}( A_{\tilde{\tau} _i} ) \subset R^{[m+2]}$, for $a \leq i \leq q$
\item $\omega _a = \tilde{\omega} _a$
\end{enumerate}
\end{claim}

\begin{proof}[Proof of Claim \ref{inductionclaim}]
First, suppose $\rho(\tau _i) \geq \tau _{i+1}$ for $a \leq i \leq q$.  Then all we need to show is that $\Phi _i \in \EA _2 ^{\rho _i (\tau _i)} (R^{[m]})$.  Note that by Corollary \ref{earhotau} it is equivalent to show that $\Phi _i ^\prime := \rho _i \circ \Phi _i \circ \rho _i ^{-1} \in \EA _2 ^{\tau _i}(R)$.

Without loss of generality, write $\Phi _i ^\prime = (z_1+x^{-s}P(x^{t_{i,2}}z_2),z_2)$ for some $P(x^{t_{i,2}}z_2) \in A_{\tau _i}[\hat{z}_1] \setminus xA_{\tau _i}[\hat{z}_1]$.  
Then
\begin{align*}
(\rho _i \circ \Phi _i \circ \omega _{i+1})(x^{t_{i,1}}z_1)&=
(\Phi _i ^\prime \circ \rho _i \circ \omega _{i+1})(x^{t_{i,1}}z_1) \\
 &= (\rho _i \circ \omega _{i+1})(x^{t_{i,1}}z_1)+x^{t_{i,1}-s}P((\rho _i \circ \omega_{i+1})(x^{t_{i,2}}z_2))
\end{align*}
Since $\rho _i (\tau _i) \geq \tau _{i+1}$, we have $\rho _i (A_{\tau _i})=A_{\rho _i(\tau _i)} \subset A_{\tau _{i+1}}$.  In particular, $(\rho _i \circ \omega _{i+1})(A_{\tau _i}) = \omega _{i+1}(A_{\rho _i (\tau _i}))  \subset \omega _{i+1}(A_{\tau _{i+1}}) \subset R^{[m+2]}$.   As $(\rho _i \circ \Phi _i \circ \omega _{i+1})(A_{\tau _i}) \subset R^{[m+2]}$, this implies that $x^{t_{i,1}-s}P((\rho _i \circ \omega _{i+1})(x^{t_{i,2}}z_2)) \in R^{[m+2]}$ as well.  Thus $t_{i,1} \geq s$ since $P \notin (x)$ and $(\rho _i \circ \omega _{i+1})(x^{t_{i,2}}z_2) \in R^{[m+2]} \setminus xR^{[m+2]}$ (by Lemma \ref{technicallemma}); but $t_{i,1} \geq s$ is precisely the condition that $\Phi _i \in \EA _2 ^{\tau _i}(R^{[m]})$ as required.    

It now suffices to assume  that $a\leq q$ is maximal such that $\rho _a (\tau _a) < \tau _{a+1}$; then $\rho _i (\tau _i) \geq \tau _{i+1}$ and $\Phi _i \in \EA _2 ^{\rho _i (\tau _{i})}(R^{[m]})$ (by the above argument) for $a+1 \leq i \leq q$.

We proceed by induction downwards on $q-a$. 
Let $b>a$ be minimal such that $\rho_b(\tau _{b}) > \tau _{b+1}$.   That is, $\rho _i (\tau _i) = \tau _{i+1}$ for $a < i < b$.  We will show that we can replace $\alpha _a \circ \rho _a \circ \Phi _a \circ \cdots \circ \alpha _b \circ \rho _b \circ \Phi _b$ by a single $\alpha _a ^\prime \circ \rho _a ^\prime \circ \Phi _a ^\prime$; then the induction hypothesis will imply that $\omega _a$ is in the desired form.  Note that by Lemma \ref{PhiEAtau} (with $\sigma = \rho _a(\tau _a)$, $\tau =\tau _{a+1}$, and $\omega=\omega _{a+1}$), we must have $\Phi _a \in \EA _2 ^{\tau _{a+1}}(R^{[m]})$.  Also, since $\rho _i (\tau _i)=\tau _{i+1}$ for $a < i < b$, $\Phi _i \in \EA _2 ^{\tau _{i+1}}(R^{[m]})$ for $a<i<b$.  Then by Lemma \ref{NoAlphaNoRho}, it suffices to assume that $\alpha _{a+1} = \cdots = \alpha _b = \id$, $\rho _{a+1}= \cdots = \rho _b = \id$, $\rho _a(\tau _a) < \tau _{a+1}=\cdots= \tau _b > \tau _{b+1}$ and $\Phi _i \in \EA _n ^{\tau _b}(R^{[m]})$ for $ a \leq i \leq b$.

A priori, it seems we may no longer be able to assume the minimality condition on the $\tau _i$ when $a < i < b$.  However, we may simply replace the $\tau _i$ by the minimal $\tau _i$ such that $\omega _i (A_{\tau _i}) \subset R^{[m+2]}$ (for $a \leq i \leq b$).  Then we may need to increase $a$ (but it will not exceed $b$) such that we may still assume $\alpha _{a+1}=\cdots=\alpha _b = \id$, $\rho _{a+1}=\cdots=\rho _b = \id$, $\Phi _{a}, \ldots, \Phi _b \in \EA _n ^{\tau _b}(R^{[m]})$, and
$$\rho _a (\tau _a) < \tau _{a+1} = \cdots = \tau _b > \tau _{b+1}$$

We also now see that
\begin{equation}\label{omegaa}
\omega _a = \alpha _a \circ \rho _a \circ \Phi _a \circ \Phi _{a+1} \circ \cdots \circ \Phi _b \circ \omega _{b+1}
\end{equation}

Set \begin{equation}\tau _b = (t_1,t_2)\end{equation} for some $t_1, t_2 \in \IN$.  Without loss of generality, assume $\Phi _a$ is elementary in $z_1$.  Then since $\rho _a(\tau _a) < \tau _{a+1}=\tau _b = (t_1,t_2)$, the minimality of $\tau _a$ implies $(\Phi _a \circ \cdots \circ \Phi _b \circ \omega _{b+1})(x^{t_1}z_1) \in xR^{[m+2]}$.  Then, by Lemma \ref{onlyOnePhi}, we may assume that $\Phi _a \circ \cdots \circ \Phi _b = \alpha \circ \rho \circ \Phi$ for some $\alpha \in \IA _2 ^{\tau _{a+1}}(R^{[m]})$, $\rho \in \GP_2 ^{\tau _{a+1}} (R^{[m]})$ and elementary $\Phi \in \EA_2  ^{\tau _{a+1}} (R^{[m]}) \cap \GL _2 ^{\tau _{a+1}}(R^{[m]})$.  Then
$$\omega _a = \alpha _a \circ \rho _a \circ \alpha \circ \rho \circ \Phi \circ \omega _{b+1}$$
Noting that $\rho _a (\tau _a) < \tau _{a+1}$, by Corollary \ref{alpharhoalpharho}, we have $\alpha _a \circ \rho _a \circ \alpha \circ \rho = \alpha _a ^\prime \circ \rho _a ^\prime \circ \Phi ^\prime$ for some $\alpha _a ^\prime \in \IA _{m+2} ^{\tau _a} (R)$, $\rho _a ^\prime = \rho _a \circ \rho \in \GP_2(S^{[m]})$, and $\Phi ^\prime \in \EA _2 ^{\tau _{a+1}}(R^{[m]})$ (since $\rho(\tau _{a+1})=\tau_{a+1}$).  Thus we have 
\begin{equation*}
\omega _a = \alpha _a ^\prime \circ \rho _a ^\prime \circ \Phi ^\prime \circ \Phi \circ \omega _{b+1} 
\end{equation*}

First, suppose $\Phi ^\prime$ and $\Phi$ are both elementary in the same variable; then we may set $\tilde{\Phi} = \Phi ^\prime \circ \Phi$ and $\tilde{\Phi} \in \EA _2 ^{\tau _{a+1}}(R^{[m]})$ is elementary, and
\begin{equation}
\omega _a = \alpha _a ^\prime \circ \rho _a ^\prime \circ \tilde{\Phi} \circ \omega _{b+1} \label{omegaa1}
\end{equation}

Similarly, if we suppose instead that $\Phi ^\prime$ and $\Phi$ are elementary in different variables, then since $\Phi \in \GL _2 ^{\tau _{a+1}}(R^{[m]})$, by Lemma \ref{acnonzero} there exist $\tilde{\rho} \in \GP _2 ^{\tau _{a+1}}(R^{[m]})$ and $\tilde{\Phi} \in \EA_2 ^{\tau _{a+1}}(R^{[m]})$ such that $\Phi ^\prime \circ \Phi=\tilde{\rho} \circ \tilde{\Phi}$.  Then we  have

\begin{equation}
\omega _a = \alpha _a ^\prime \circ (\rho _a ^\prime \circ \tilde{\rho}) \circ \tilde{\Phi} \circ \omega _{b+1} \label{omegaa2}
\end{equation}

Note that since $\tilde{\rho} \in \GP _2 ^{\tau _{a+1}}(R^{[m]})$ that $\rho _a ^\prime (\tau _a) < \tau _{a+1}$ implies $(\rho _a ^\prime \circ \tilde{\rho})(\tau _a) < \tau _{a+1}$.  Thus, in either case, we see we have written $\omega _a$, which was originally \eqref{omegaa} a product of $q-a+1$ elementaries, in \eqref{omegaa1} or \eqref{omegaa2} as a product with $q-b$ elementaries, and since $a<b$, we must have $q-a+1>q-b$.  The induction hypothesis then completes the proof.
\end{proof}

\section*{\small Acknowledgements}The author would like to thank David Wright and Brady Rocks for helpful discussions and feedback.

\bibliographystyle{siam}
\bibliography{ref}

\end{document}